\documentclass[12pt,reqno]{amsart}
\usepackage[margin=1in,letterpaper]{geometry}
\usepackage{graphicx}
\usepackage{amssymb}
\usepackage{amsthm}
\usepackage{mathrsfs}
\usepackage{enumitem}
\usepackage{hyperref}
\usepackage{tikz}

\makeatletter\let\over\@@over\makeatother

\numberwithin{equation}{section}
\theoremstyle{plain}
\newtheorem{theorem}{Theorem}[section]
\newtheorem{proposition}[theorem]{Proposition}

\newtheorem{lemma}[theorem]{Lemma}
\theoremstyle{remark}
\newtheorem{remark}[theorem]{Remark}

\newenvironment{proof1}%
{\begin{trivlist} \item[]{{\em Proof} }}%
{\hspace*{\fill}$\Box$\end{trivlist}}

\newcommand{\be}{\begin{equation}}
\newcommand{\ee}{\end{equation}}
\newcommand{\bse}{\begin{subequations}}
\newcommand{\ese}{\end{subequations}}

\newcommand{\LV}{\left|}
\newcommand{\RV}{\right|}
\newcommand{\LN}{\left\|}
\newcommand{\RN}{\right\|}
\newcommand{\LB}{\left[}
\newcommand{\RB}{\right]}
\newcommand{\LC}{\left(}
\newcommand{\RC}{\right)}

\newcommand{\LCB}{\left\{}
\newcommand{\RCB}{\right\}}
\newcommand{\ub}{\overline{u}}

\newcommand{\ux}{\overline{u_x}}
\newcommand{\uxs}{\overline{u_x^2}}

\newcommand{\R}{\mathbb{R}}
\newcommand{\Rr}{\mathcal{R}}
\newcommand{\placeholder}{\,\cdot\,}

\newcommand{\Q}{\mathcal{Q}}
\newcommand{\Pa}{\mathcal{P}}
\newcommand{\M}{\mathcal{M}}

\begin{document}

\title[Instability of peakons in the Novikov equation]{$W^{1,\infty}$ instability of $H^1$-stable peakons \\in the Novikov equation}
\date{\today}

\author[R.M. Chen]{Robin Ming Chen}
\address{Department of Mathematics, University of Pittsburgh, Pittsburgh, PA 15260}
\email{mingchen@pitt.edu}

\author[D.E. Pelinovsky]{Dmitry E. Pelinovsky}
\address{Department of Mathematics and Statistics, McMaster University, Hamilton, ON L8S 4K1, Canada}
\email{dmpeli@math.mcmaster.ca}

\begin{abstract}
It is known from the previous works that the peakon solutions of the Novikov equation are orbitally and asymptotically stable in $H^1$.
We prove, via the method of characteristics, that these peakon solutions are unstable under $W^{1,\infty}$-perturbations.
Moreover, we show that small initial $W^{1,\infty}$-perturbations of the Novikov peakons can lead to the finite time blow-up of
the corresponding solutions.
\end{abstract}


\maketitle

\section{Introduction}

The integrable Novikov equation
\be\label{novikov}
 u_t-u_{xxt}+4u^2u_x=3uu_xu_{xx}+u^2u_{xxx}
\ee
is proposed by Novikov \cite{Novikov09} from a Lie symmetry analysis of nonlocal partial differential equations.
Reformulating \eqref{novikov} in terms of the momentum density $m = u - u_{xx}$ yields the following evolution form
\be\label{novikov m}
m_t+u^2m_x+\frac 32 (u^2)_x m = 0.
\ee
Hence, this Novikov equation can be regarded as a cubic nonlinear generalization of the Camassa--Holm (CH) equation \cite{CH93}
(derived earlier in \cite{FF81}):
\be\label{ch m}
m_t + u m_x + 2 u_x m = 0.
\ee

The Novikov equation shares many common analytical properties with the CH equation.
It belongs to the class of completely integrable equations thanks to the existence of
the Lax pair \cite{HW08,Novikov09} and the bi-Hamiltonian structure \cite{HW08}.
The Novikov equation can exhibit the phenomenon of wave-breaking \cite{BUNovikov}
(see also recent work in \cite{Chen16JFA}). Another remarkable feature of the Novikov equation
 is the existence of peaked traveling wave solutions (called \emph{peakons}):
\be\label{peakon}
u(t,x) = \varphi_c(x - ct - x_0), \quad c > 0, \;\; x_0 \in \R
\ee
with
\be
\varphi_c(x) = \sqrt{c} e^{-|x|}, \quad x \in \R,
\ee
with corner singularities at the peaks \cite{GX09,HLS09,HW08}. In what follows, we will be dealing with the peakons
propagating with the unit speed, for which we denote $\varphi := \varphi_{c = 1}$.

\subsection{Previous works}

The (local) well-posedness theory for strong solutions to the Novikov equation \eqref{novikov}
is a well-studied subject \cite{HH12,NZ11,T11,WG16,WY12}. However, these results are not applicable 
to the scopes of our work since we have to consider weak solutions due to the wave breaking occurrence and the presence of peakons.

The Novikov equation \eqref{novikov} can be rewritten in the convolution form
\be\label{novikov weak}
u_t + u^2 u_x + (1 - \partial_x^2)^{-1} \partial_x \LC {3\over2} uu_x^2 + u^3 \RC + (1 - \partial_x^2)^{-1} \LC {1\over2} u_x^3 \RC = 0,
\ee
which suggests $H^1 \cap W^{1,3}$ as a natural space for weak solutions. It turns out that,
by incorporating one of the conservation laws
\be\label{cons law 1}
E(u) := \int_\R \LC u^2  + u_x^2 \RC \,dx,
\ee
the existence and uniqueness of global weak solutions can be established in
$H^1 \cap W^{1,\infty}$ under an additional constraint on the initial datum
$u_0$ that $m_0 := u_0 - u_{0xx}$ is a positive Radon measure \cite{WG16,WY11}.
The sign condition $m_0 \ge 0$ was replaced by $u_0 \ge 0$ in \cite{L13} and a weak solution in $H^1 \cap W^{1,4}$
with the one-sided $L^\infty$ bound on the gradient of $u$ is obtained through a viscous approximation,
at the price of losing the conservation of $E$ and hence the uniqueness of solutions.

If another conservation law
\be\label{cons law 2}
F(u) := \int_\R \LC u^4 + 2u^2u_x^2 - {1\over3} u_x^4 \RC \,dx
\ee
is taken into account, the global weak solution theory can be casted in $H^1 \cap W^{1,4}$
without any restrictions on the initial datum \cite{Chen18IUMJ}.
The data-to-solution map is shown to be Lipschitz continuous on bounded sets of $H^1 \cap W^{1,4}$
under an optimal transport metric \cite{Chen18ARMA}.

The importance of the two conservation laws $E(u)$ and $F(u)$ is also manifested
in the stability analysis of the peakons. In \cite{NovikovStab}, a Lyapunov function was constructed from the two conserved
quantities, through which an $H^1$-orbital stability of peakons was established.  Among various assumptions
on the initial perturbation $u_0 \in H^s$ with $s \geq 3$, a crucial one in \cite{NovikovStab}
was positivity of $m_0 := u_0 - u_{0xx}$.
Such a sign property is preserved in the time evolution of the Novikov equation, from which one can
control $|u_x(t,x)| \le |u(t,x)| \leq E(u_0)$,
leading to a global solution in $H^s$, $s \geq 3$.
The same sign condition is a key to the construction of the Lyapunov function for peakons in \cite{NovikovStab}.

Applying this orbital stability and utilizing the finite speed propagation property,
an $H^1$-asymptotic stability of the Novikov peakon was obtained in \cite{ChenAsympt}
for the initial datum $u_0 \in H^1$ with $m_0$ being a nonnegative Radon measure.

The sign condition on $m_0$, and hence the boundedness of $|u_x(t,x)|$, presents a serious obstacle
in the analysis of $W^{1,\infty}$-instability of peakons and might even exclude this kind of
instability. Therefore, for our work we need an $H^1$ orbital stability result
for the initial datum without the sign condition on $m_0$. In a recent work \cite{Chen19},
such a sign constraint was removed, at the price that the global strong solutions in \cite{NovikovStab}
were replaced by the local strong solutions. The following theorem records the corresponding
result from \cite{Chen19}.\\

{\bf Theorem A}($H^1$-orbital stability)
{\it For every $0 < \varepsilon \ll 1$ and for every $u_0\in H^s(\R)$ with $s > 5/2$ satisfying
\begin{equation*}
\|u_0 - \varphi\|_{H^1} < \varepsilon^4,
\end{equation*}
the corresponding solution $u \in C([0,T),H^s)$ to the Novikov equation \eqref{novikov} with initial datum $u_0$
and the maximal existence time $T > 0$ satisfies
\begin{equation*}
\sup_{t\in [0, T)} \|u(t,\placeholder) - \varphi(\placeholder - \xi(t))\|_{H^1} < 2\LC 4 + \|u_{0x}\|_{L^\infty}^{1/2} \RC \varepsilon 
\end{equation*}
where $\xi(t)$ is a point of maximum of $u(t, \placeholder)$. }

Theorem A only considers smooth solutions, whereas for our instability argument we need to control
the evolution of solutions that are only Lipschitz. For this purpose, we need to reexamine the
$H^1$ stability in a weaker regularity framework, which we do in Theorem \ref{thm stab weak}.

\subsection{Main results and methodology}

The purpose of the current work is to understand the stability of peakons in the Novikov equation
under the $W^{1,\infty}$ perturbations which preserve the original smoothness of peakons. In particular,
we will consider piecewise $C^1$ perturbations to a single peakon and study their evolution
under both the linearized and nonlinear flows associated to the Novikov equation \eqref{novikov}.
As is formulated in the following two theorems, we will prove that piecewise $C^1$ perturbations to
a single peakon may grow in the $W^{1,\infty}$ norm in spite of being bounded in the $H^1$ norm
both in the linearized and nonlinear flows.

First we derive in Section \ref{subsec_deriv lin} the Cauchy problem for the {\it linearized} evolution of a perturbation $v(t,x)$ to the peakon $\varphi(x)$
in the form
\be\label{lin sim v}
\begin{cases}
v_t + (\varphi^2 - 1) v_x + \varphi_x \LB v(t,0) - \varphi v \RB = 0, \\
v|_{t = 0} = v_0,
\end{cases}
\ee
which, following the idea of \cite{DmitryPreprint}, motivates us to work in the space $C^1_0 \subset W^{1,\infty}$ defined as
\be\label{space C}
C^1_0 := \LCB v \in C(\R) \cap C^1(\R^+) \cap C^1(\R^-): \quad v, v_x \in L^\infty \RCB.
\ee
Hence $v_0 \in C^1_0$ may have at most one peak at $x = 0$, which is also a location of the peak
of $\varphi$. The method of characteristics can thus be implemented to provide an explicit solution to \eqref{lin sim v} in $H^1 \cap C^1_0$, allowing one to obtain the following result.

\begin{theorem}[Linear instability]\label{thm lin stab}
For any given initial datum $v_0 \in H^1 \cap C^1_0$, there exists a unique global solution
$v \in C(\R, H^1 \cap C^1_0)$ to the linearized problem \eqref{lin sim v} such that
\be\label{lin H^1 cons}
\|v(t, \cdot)\|^2_{H^1(\R^\pm)} = \|v_0\|^2_{H^1(\R^\pm)} \qquad \textnormal{(linear $H^1$ stability)}
\ee
and
\be\label{lin W growth}
\|v_x(t, \cdot)\|_{L^\infty(\R^+)} \ge |v_0(0) + v_{0x}(0^+)| e^t - |v_0(0)| \quad \textnormal{(linear $W^{1,\infty}$ instability)}
\ee
for all $t > 0$.
\end{theorem}

The nonlinear analysis is more delicate. The Cauchy problem for the Novikov equation can be formulated as
\be\label{Cauchy nov}
\begin{cases}
u_t + u^2u_x + \Q[u] = 0, \qquad t>0\\
u(0,x) = u_0(x),
\end{cases}
\ee
where
\be\label{def Q}
\Q[u] := {1\over2}\varphi_x \ast \LC {3\over2} u u_x^2 + u^3 \RC + {1\over4}\varphi \ast u_x^3.
\ee
Similarly as in the linear analysis, we would like to first establish a well-posedness theory of the evolution of the perturbation $v$ in $H^1 \cap C^1_0$. Compared with the Camassa--Holm case \cite{DmitryPreprint}, the Cauchy problem for $v$ in $H^1$ given by \eqref{nonlinear v}
was not studied before, hence we cannot use the previous well-posedness results. By a careful retooling of the method
of characteristics, the Cauchy problem for $v$ can be transformed to a dynamical system \eqref{nonlinear-F} where the vector field on the right-hand side consists of local terms of polynomial type and nonlocal terms that can be shown to be locally Lipschitz. Hence standard ODE theory applies to imply local well-posedness if solutions for $v$ in $H^1 \cap C^1_0$ established in Theorem \ref{theorem GWP}.

The $H^1$ orbital stability result (Theorem A) suggests that in order for the peakons to be $W^{1,\infty}$-unstable, it is necessary to track the dynamics of the gradient $v_x$ of the perturbation and look to show that $\|v_x\|_{L^\infty}$ exhibit substantial growth. However Theorem A only treats strong solutions, and therefore a similar result in the weak solution framework is needed and is established in Theorem \ref{thm stab weak}.

The key ingredient in proving the $H^1$ orbital stability is to construct a Lyapunov function using
the two conservation laws $E$ and $F$ similar to what is done in \cite{Chen19}. For strong solutions,
the conservation laws can be easily checked by utilizing the bi-Hamiltonian structure of the equation.
However for weak solutions this becomes more delicate. Our strategy is based on regularizing the system
and commuting the regularization with nonlinearity. The conservation laws can then be realized by deriving
crucial commutator estimates in order to show that the remainder terms converge to zero as the regularization
parameter tends to zero as is done in Lemma \ref{lem cons weak}.

It turns out that the dynamics of $v_x$ simplifies when restricted at the peak location, see equation \eqref{W-0-dynamics}.
The corresponding differential equation consists of a Ricatti-like term, the terms that involve interaction with $v$, and a nonlocal term. The orbital stability ensures that all the interaction terms are small. Another important consequence of the orbital stability is that the nonlocal term is also small. This way a Ricatti-type inequality can be obtained, which in turn leads to a finite time blow-up.
\begin{theorem}[Nonlinear instability]\label{thm nonlin stab}
For every $\delta > 0$, there exist $t_0 > 0$ and $u_0 \in H^1 \cap C^1_0$
satisfying
\begin{equation}
\label{initial-bound-theorem}
\| u_0 - \varphi \|_{H^1} + \| u_{0x} - \varphi_x \|_{L^{\infty}} < \delta,
\end{equation}
such that the unique solution $u \in C([0,T),H^1 \cap W^{1,\infty})$
to the Cauchy problem (\ref{Cauchy nov}) with the initial datum $u_0$ and
the maximal existence time $T > t_0$ satisfies
$u(t,\cdot + a(t)) \in C^1_0$ for $t \in [0,T)$ and
\begin{equation}
\label{final-bound-theorem}
\| u_x(t_0,\cdot) - \varphi_x(\cdot - a(t_0)) \|_{L^{\infty}} >1,
\end{equation}
where $a(t)$ is a point of peak of $u(t,\cdot)$ for $t \in [0,T)$ such that $a(0) = 0$. Moreover, there exist initial datum $u_0$ satisfying \eqref{initial-bound-theorem} such that $T<\infty$ for the corresponding solution $u \in C([0,T),H^1 \cap W^{1,\infty})$.
\end{theorem}

\begin{remark}
The results of Theorems \ref{thm lin stab} and \ref{thm nonlin stab} are very similar to the results found
in \cite{DmitryPreprint} for the CH equation (\ref{ch m}) except that the $H^1$ norm
of the peaked perturbation grows in the linear evolution of the CH equation, whereas the $H^1$ norm
does not grow for the linearized Novikov equation. The discrepancy between the two results
confirm the previous intuition \cite{constantin2001}
that the linearized evolution in $H^1$ does not imply anything for the nonlinear evolution of
the quasilinear equations with peakons and wave breaking.
\end{remark}
\begin{remark}
An interesting outcome of our instability theorem is that it provides a new way to generate wave breaking in the {\it weak} solution setting. To the best of the authors' knowledge, so far the vast literature on the blow-up analysis for quasilinear integrable equations, like the Camassa--Holm equation \cite{Brandolese2014CMP,CH93,CE98,CL09,FF81}, the Degasperis-Procesi equation \cite{Chen16JFA,DP99,ELY06,LY06}, and the Novikov equation \cite{Chen16JFA,BUNovikov}, is performed in the framework of strong solutions. It is plausible that the idea used here can be extended to other peakon models.
\end{remark}

\section{Linear analysis}\label{sec_linear}

Here we investigate the linear stability of peakons and prove Theorem \ref{thm lin stab}. For simplicity,
we consider a single peakon (\ref{peakon}) traveling with the unit speed $c = 1$
and denote it by $\varphi(x) \equiv \varphi_{c=1}(x) = e^{-|x|}$.
Note that ${1\over2} \varphi(x)$ is the Green's function of $1 - \partial_x^2$ on $\R$, that is,
\be\label{green}
(1 - \partial_x^2) \varphi = 2 \delta_0, \qquad (1 - \partial_x^2)^{-1} f = {1\over 2} \varphi \ast f.
\ee
Some further properties of $\varphi$ are given by
\begin{align}\label{squares}
\varphi_x^2(x) = \varphi^2(x), \quad x \in \mathbb{R} \backslash \{0\}
\end{align}
and
\begin{align}
\|\varphi\|_{L^2} = \|\varphi\|_{L^2} = \|\varphi\|_{L^\infty} = \|\varphi_x\|_{L^\infty} = 1. \label{norms}
\end{align}
In what follows, we derive the linearized problem (\ref{lin sim v}), solve it by means of characteristics,
and finally obtain relevant estimates for the proof of Theorem \ref{thm lin stab}.

\subsection{Derivation of the linearized problem}\label{subsec_deriv lin}

To study the linearization of \eqref{novikov weak} around $\varphi$, we decompose $u(t,x)$
as the sum of a modulated peakon and its perturbation $v$ in the form:
\be\label{decomp}
u(t,x) = \varphi(x - a(t)) + v(t, x - a(t)).
\ee
The stationary equation for peakon $\varphi$ is defined for every $x \neq 0$ in the form:
\be
\label{stat-peakon}
(\varphi^2 - 1) \varphi' + \Q[\varphi] = 0,
\ee
where $\Q$ is given by (\ref{def Q}). When we plug in (\ref{decomp}) and (\ref{stat-peakon})
into (\ref{novikov weak}) and truncate at the linear terms in $v$, we
obtain the linearized equation for $v$ in the form:
\be\label{lin v}
(1-\dot a) \varphi_x + v_t - \dot a v_x + (\varphi^2 v)_x + {3\over2}\varphi_x \ast \LC  \varphi^2 v + {1\over2} \varphi_x^2 v + \varphi \varphi_x v_x \RC + {3\over4} \varphi \ast \LC \varphi_x^2 v_x \RC = 0.
\ee
The following proposition allows us to simplify the nonlocal terms in \eqref{lin v}
and write it in the local form (\ref{lin sim v}).

\begin{proposition}\label{prop sim v}
For $v\in H^1$ we have
\be\label{sim nonlocal}
{3\over2}\varphi_x \ast \LC  \varphi^2 v + {1\over2} \varphi_x^2 v +
\varphi \varphi_x v_x \RC + {3\over4} \varphi \ast \LC \varphi_x^2 v_x \RC = 3 \varphi_x \LB v(0)- \varphi v \RB.
\ee
\end{proposition}

\begin{proof}
By using \eqref{squares} and integrating by parts, we obtain
\[
{3\over4} \varphi \ast \LC \varphi_x^2 v_x \RC = {3\over4} \varphi \ast \LC \varphi^2 v_x \RC = {3\over4} \varphi_x \ast \LC \varphi^2 v \RC - {3\over4} \varphi \ast \LB (\varphi^2)_x v \RB.
\]
From \eqref{green} we see that $\varphi_{xx} = \varphi - 2\delta_0$, and hence further using $\varphi(0) = 1$
and integrating by parts, we obtain
\begin{align*}
{3\over2}\varphi_x \ast (\varphi \varphi_x v_x) & = {3\over 4} \varphi_x \ast \LB (\varphi^2)_x v_x \RB = {3\over 4} \varphi_{xx} \ast \LB (\varphi^2)_x v \RB - {3\over4} \varphi_x \ast \LB (\varphi^2)_{xx} v \RB \\
& = {3\over4} \varphi \ast \LB (\varphi^2)_x v \RB - {3\over2} (\varphi^2)_x v - 3 \varphi_x \ast \LB (\varphi^2 - \delta_0\varphi) v \RB \\
& = {3\over4} \varphi \ast \LB (\varphi^2)_x v \RB - 3 \varphi_x \ast \LC \varphi^2 v \RC - 3 \varphi \varphi_x v + 3 \varphi_x v(0).
\end{align*}
Substituting the two representations into the left-hand side of (\ref{sim nonlocal}) completes the proof of the proposition.
\end{proof}

From Proposition \ref{prop sim v} we rewrite the nonlocal term in \eqref{lin v} in the local form:
\be\label{lin v rewrite}
(1-\dot a) \varphi_x + v_t - \dot a v_x + (\varphi^2 v)_x + 3 \varphi_x \LB v(t,0)- \varphi v \RB = 0,
\ee
where if $v \in C(\mathbb{R})$, then the last term is continuous everywhere including $x = 0$ thanks to $\varphi(0) = 1$.
Since $\varphi_x$ is continuous everywhere except at the origin, the other terms of the linearized equation \eqref{lin v rewrite}
are continuous at $x = 0$ if
\begin{equation}
\label{dot-a-linear}
\dot a(t) = 1 + 2 v(t,0) + \mathcal{O}(v(t,0)^2)
\end{equation}
where the remainder term in (\ref{dot-a-linear}) is truncated at the linear approximation.
Plugging \eqref{dot-a-linear} into \eqref{lin v rewrite} and keeping only the linear terms in $v$,
we finally obtain the Cauchy problem (\ref{lin sim v}) for the linearized equation at a single peakon.

\subsection{Solution to the linearized problem}\label{subsec solve lin}

Following the idea of \cite{DmitryPreprint}, we will solve the linearized problem \eqref{lin sim v} using the method of characteristics. For this, we first define the characteristic curves $q(t,s)$ as
\be\label{char}
\begin{cases}
\displaystyle {d q \over dt} = \varphi^2(q) - 1, \\
q(0, s) = s.
\end{cases}
\ee
For any fixed $s\in \R$, the initial-value problem (\ref{char}) has a unique solution
since $\varphi$ is Lipschitz. Moreover, it follows that
\be \label{diffeo}
q_s(t,s) = \exp \LC\int^t_0 2\varphi \varphi_x (q(\tau, s)) \, d\tau \RC > 0
\ee
hence $q(t,\cdot)$ is a diffeomorphism on $\R$ for any $t \in \mathbb{R}$.

Since $\varphi(0) = 1$, we have $q(t, 0) = 0$ for any $t\in \R$, meaning that
the location of the peak of $\varphi$ is invariant under the flow of system \eqref{char}.
Solving \eqref{char} explicitly, we obtain that
\be\label{soln char}
q(t, s) = \left\{\begin{array}{ll}
\displaystyle {1\over2}\log\LB 1 + \LC e^{2s} - 1 \RC e^{-2t} \RB, \ & \ s>0, \\\\
0, &\ x = 0,\\\\
\displaystyle-{1\over2}\log\LB 1 + \LC e^{-2s} - 1 \RC e^{2t} \RB, \ & \ s<0.
\end{array}\right.
\ee
From \eqref{soln char} it follows that $q(t,s) \to 0$ as $s \to 0^{\pm}$.
Define
\be
\label{V-on-characteristics}
V(t,s) := v(t, q(t,s)).
\ee
From \eqref{soln char} we know that when solving \eqref{lin sim v} along the characteristics $q$,
we can consider characteristics with $s>0$ separately from characteristics with $s<0$. This corresponds
to partitioning of $\R$ into $\R^+$ and $\R^-$ in the physical space and suggests us to consider solutions
$v(t,\cdot) \in H^1 \cap C^1_0$ for any $t \in \mathbb{R}$, where
$C^1_0 \subset W^{1,\infty}$ is given by (\ref{space C}).
It follows from (\ref{lin sim v}) and (\ref{char}) that $V(t,s)$ satisfy
\be\label{char-V}
\begin{cases}
\displaystyle {d V \over dt} = \varphi_x(q) \left[ \varphi(q) V - V(t,0) \right], \\
V(0, s) = v_0(s),
\end{cases}
\ee
where we have used that $V(t,0) = v(t,q(t,0)) = v(t,0)$.
It follows from (\ref{char-V}) as $s \to 0^+$ that if $V(t,\cdot) \in C(\mathbb{R})$ for $t \in \mathbb{R}$,
then $V(t,0) = V(0,0) = v(0,0) = v_0(0)$.
Therefore, for $s > 0$ we are solving
\be\label{ODE V}
\begin{cases}
\displaystyle {dV \over dt} = -e^{2q(t,s)} V + e^{-q(t,s)}v_0(0), \\
V(0,s) = v_0(s).
\end{cases}
\ee
Direct computation yields the unique solution to the initial-value problem (\ref{ODE V}) in the form:
\be\label{soln V+}
V(t,s) = {v_0(s) + v_0(0)(e^t - 1) e^{-s} \over \sqrt{1 + (e^{2t} -1) e^{-2s}}}, \quad s>0.
\ee
Clearly we see that $\displaystyle \lim_{s\to 0^+} V(t,s) = v_0(0)$.
Similarly, for $s<0$ we obtain the unique solution in the form:
\be\label{soln V-}
V(t,s) = {v_0(s) - v_0(0)(1 - e^{-t}) e^s \over \sqrt{1 - (1 - e^{-2t}) e^{2s}}}, \quad s < 0,
\ee
satisfying $\displaystyle \lim_{s\to 0^-} V(t,s) = v_0(0)$.

One can also compute explicitly the evolution of $v_x$ along the characteristics. Define
\be
\label{W-on-characteristics}
W(t,s) := v_x(t, q(t,s)).
\ee
Chain rule implies that
\be\label{relation W}
W(s,t) = {V_s(t,s) \over q_s(t,s)}.
\ee
From \eqref{soln char}, \eqref{soln V+}, and \eqref{relation W} we obtain that
\be\label{soln W+}
\begin{split}
W(t,s) = & \sqrt{1 + (e^{2t}-1) e^{-2s}} \LB v'_0(s) - v_0(0) (e^t -1) e^{-s} \RB \\
& + {(e^{2t}-1)e^{-2s}\LB  v_0(s) + v_0(0) (e^t -1) e^{-s} \RB \over \sqrt{1 + (e^{2t}-1)e^{-2s}}}, \quad s>0.
\end{split}
\ee
It follows from (\ref{soln W+}) as $s \to 0^+$ that
\be\label{grow 0+}
\lim_{s\to 0^+} W(t,s) = v_0(0)(e^t - 1) + v_0'(0^+) e^t.
\ee
Hence, the gradient $\lim_{x \to 0^+} v(t,x) = \lim_{s \to 0^+} W(t,s)$ grows exponentially in time.
Similarly, from \eqref{soln char}, \eqref{soln V-}, and \eqref{relation W} we obtain that
\be\label{negative s}
\begin{split}
W(t,s) = & \ \sqrt{1 + (e^{-2t}-1) e^{2s}} \LB v'_0(s) - v_0(0) (1 - e^{-t}) e^{s} \RB \\
& + {(1 - e^{-2t}) e^{2s}\LB  v_0(s) - v_0(0) (1 - e^{-t}) e^{s} \RB \over \sqrt{1 - (1 - e^{-2t})e^{2s}}}, \quad s<0,
\end{split}
\ee
from which we obtain
\be\label{grow 0-}
\lim_{s\to 0^-} W(t,s) = v_0(0)(1 - e^{-t}) + v_0'(0^-) e^{-t}.
\ee
Hence, the gradient $\lim_{x \to 0^-} v(t,x) = \lim_{s \to 0^-} W(t,s)$ decays exponentially in time.

The following lemma justifies the solution
constructed in \eqref{soln char}, \eqref{soln V+}, \eqref{soln V-}, \eqref{soln W+},
and \eqref{negative s} and provides useful estimates.

\begin{lemma}\label{lem lin problem}
For any $v_0 \in H^1 \cap C^1_0$, the Cauchy problem \eqref{lin sim v} admits
a unique global solution $v \in C(\R; H^1 \cap C^1_0)$ satisfying the estimates:
\begin{align}
& \|v(t,\cdot)\|_{L^\infty(\R^+)} \le |v_0(0)| + \|v_0\|_{L^\infty(\R^+)}, \label{infty bound plus}\\
& \|v_x(t, \cdot)\|_{L^\infty(\R^+)} \ge |v_0(0) + v'_0(0^+)| e^t - |v_0(0)|, \label{deriv infty lower bound}
\end{align}
for any $t > 0$.
\end{lemma}

\begin{proof}
If $v_0 \in H^1 \cap C^1_0$, then the solution in (\ref{soln V+}) and (\ref{soln V-})
satisfies $V(t,\cdot) \in H^1 \cap C^1_0$ for any $t \in \R$ so that $V(t,\cdot)$ is locally Lipschitz continuous everywhere on $\mathbb{R}$.
By the existence and uniqueness theory for differential equations,
$V(t,s)$ is the unique solution of the initial-value problem (\ref{char-V}) in this class of functions.
Moveover, thanks to the property (\ref{diffeo}) and the property $q(t,s) \sim s$ as $|s| \to \infty$,
we have $v(t,\cdot) \in H^1 \cap C^1_0$ for any $t \in \R$.

Since $q$ is a diffeomorphism on $\R^+ \to \R^+$ and $\R^- \to \R^-$, we have
\[
\|v\|_{L^\infty(\R^\pm)} = \|V\|_{L^\infty(\R^\pm)}, \qquad \|v_x\|_{L^\infty(\R^\pm)} = \|W\|_{L^\infty(\R^\pm)}.
\]
From \eqref{soln V+} we infer that
\[
|V(t,s)| \le |v_0(s)| + |v_0(0)|, \quad s > 0,
\]
which yields (\ref{infty bound plus}). It follows from \eqref{grow 0+} that
\[
\|W(t,\cdot)\|_{L^\infty(\R^+)} \ge \lim_{s\to 0^+} |W(t,s)| \ge \LV v_0(0) + v'_0(0^+) \RV e^t - |v_0(0)|,
\]
which yields (\ref{deriv infty lower bound}).
\end{proof}

\begin{remark}
Even if $v_0 \in H^1 \cap C^1$, the solution of the linearized problem (\ref{lin sim v}) only exists
in $v(t,\cdot) \in H^1 \cap C^1_0$ because the jump of the derivative $v_x$ across $x = 0$
appears instantaneously in time:
$$
[ v_x(t,x) ]^+_- := \lim_{x \to 0^+} v_x(t,x) - \lim_{x \to 0^-} v_x(t,x) = 2 v_0(0) (\cosh t - 1) + 2 v_0'(0) \sinh t,
$$
where $v_0'(0) = \lim_{x \to 0^+} v_{0x}(x) = \lim_{x \to 0^-} v_{0x}(x)$.
\end{remark}

\subsection{$H^1$ conservation of $v$}

 Estimate \eqref{deriv infty lower bound} in Lemma \ref{lem lin problem} indicates
 the linear $W^{1,\infty}$ instability of the Novikov peakons. For the Camassa--Holm peakons
 it is showed \cite{DmitryPreprint} that the perturbation are also $H^1$ linearly unstable.
 However for Novikov peakons, we will prove that the $H^1$ norm of the linearized perturbation $v$
 satisfying \eqref{lin sim v} is conserved for all time.

\begin{lemma}
\label{lem lin-H1}
The unique global solution $v \in C(\R; H^1 \cap C^1_0)$ in Lemma \ref{lem lin problem} satisfies
\be\label{lin H^1 cons}
\|v(t, \cdot)\|^2_{H^1(\R^\pm)} = \|v_0\|^2_{H^1(\R^\pm)}
\ee
for every $t \in \mathbb{R}$.
\end{lemma}

\begin{proof}
Multiplying the linearized equation \eqref{lin sim v} by $v$ and integrating on $\R^+$ using integration by parts we have
\be\label{lin L^2 v}
{1\over2}{d\over dt} \|v\|^2_{L^2(\R^+)} - 2 \int^\infty_0 \varphi \varphi_x v^2 \,dx + v(t,0) \int^\infty_0 \varphi_x v \,dx = 0.
\ee
Differentiating \eqref{lin sim v} with respect to $x$ yields
\be \label{lin der v}
v_{xt} + (\varphi^2v_x)_x - v_{xx} + \varphi_{xx} v(0) - (\varphi \varphi_x v)_x = 0
\ee
Multiplying (\ref{lin der v} by $v_x$ and integrating over $\R^+$, we obtain
\be\label{lin L^2 v_x}
\begin{split}
{1\over2}{d\over dt} \|v_x\|^2_{L^2(\R^+)} & + \int^\infty_0 \left( (\varphi^2)_x v_x^2 + \varphi^2 v_x v_{xx} \right) \,dx
- \int^\infty_0 v_x v_{xx} \,dx + v(t,0)\int^\infty_0 \varphi v_x \,dx \\
& - \int^\infty_0 \left( 2 \varphi^2 vv_x + \varphi \varphi_x v_x^2 \right) \,dx = 0.
\end{split}
\ee
where we have used that $\varphi_{xx} = \varphi$ and $\varphi_x^2 = \varphi^2$ on $\R^+$.
Using the fact that $\varphi(0) = 1$, $\varphi_x(0^+) = -1$, we integrate by parts and simplify (\ref{lin L^2 v_x}) to the form:
\be\label{lin L^2 v_x sim}
{1\over2}{d\over dt} \|v_x\|^2_{L^2(\R^+)} + 2 \int^\infty_0 \varphi \varphi_x v^2 \,dx - v(t,0) \int^\infty_0 \varphi_x v \,dx = 0.
\ee
Adding \eqref{lin L^2 v} and (\ref{lin L^2 v_x sim}) yields
\[
{d\over dt} \|v\|^2_{H^1(\R^+)} = 0, \quad \Rightarrow \quad \|v(t,\cdot)\|^2_{H^1(\R^+)} = \|v_0\|^2_{H^1(\R^+)}, \ \text{ for all }\ t>0.
\]
Similarly we can prove the same result on $\R^-$, and hence we conclude the proof.
\end{proof}

\begin{remark}
Lemma \ref{lem lin-H1} can be proven by integrating the explicit solutions (\ref{soln V+}) and (\ref{soln W+}) on $\mathbb{R}^+$
along the characteristics (\ref{soln char}) with the chain rule:
\[
\| v(t,\cdot) \|^2_{H^1(\R^+)} = \int_0^{\infty} \left[ V(t,s)^2 + W(t,s)^2 \right] q_s(t,s) ds = \| v_0 \|^2_{H^1(\R^+)},
\]
and similarly with the explicit solutions (\ref{soln V-}) and (\ref{negative s}) on $\mathbb{R}^-$.
\end{remark}

\begin{proof1}{\em of Theorem \ref{thm lin stab}.}
Lemma \ref{lem lin problem} gives the existence of the unique solution $v \in C(\mathbb{R},H^1\cap C^1_0)$
to the linearized problem \eqref{lin sim v} for any initial datum $v_0 \in H^1 \cap C^1_0$ satisfying the estimate
\eqref{lin W growth}. Lemma \ref{lem lin-H1} gives the $H^1$ conservation \eqref{lin H^1 cons}.
\end{proof1}

\section{Nonlinear analysis}\label{sec nonlinear}

Here we investigate the nonlinear dynamics of perturbations near a single peakon and prove Theorem \ref{thm nonlin stab}.
In what follows, we review weak solutions for the Cauchy problem (\ref{Cauchy nov}),
obtain an improved version of the $H^1$-orbital stability of a single peakon compared to Theorem A,
derive the nonlinear system for peaked perturbations to a single peakon, solve this system
with the method of characteristics, and obtain relevant estimates for the proof of Theorem \ref{thm nonlin stab}.

\subsection{Weak solution theory}\label{subsec weak soln}

Let's first recall two known results for global weak solutions to the Cauchy problem of the Novikov equation (\ref{Cauchy nov}).
The first result holds for initial datum $u_0 \in H^1$ and assumes the sign condition on $m_0 := u_0 - u_{0,xx}$.

\begin{theorem}[\cite{WY11}]\label{thm weak m}
For any $u_0 \in H^1$ with $m_0 \in \M_+(\R)$, where $\M_+$ is the space of
non-negative finite Radon measures on $\R$, the Cauchy problem \eqref{Cauchy nov} admits
a unique global weak solution $u\in W^{1,\infty}(\R^+ \times \R) \cap C(\R^+; H^1(\R))$
such that $m(t,\cdot) \in \M^+(\R)$ for all $t>0$, where $m := u - u_{xx}$. Moreover, $E(u)$ and $F(u)$ are conservation laws.
\end{theorem}

\begin{remark}
The statement we give in Theorem \ref{thm weak m} is stronger than the original statement of \cite[Theorem 3.1]{WY11}. 
Firstly, the solution constructed in \cite{WY11} has weaker regularity $u \in L^{\infty}(\R^+; H^1(\R))$. 
However one can improve it to the strong topology $u \in C(\R^+; H^1(\R))$ by further using the conservation of $E(u)$. 
Secondly, \cite{WY11} only asserts the conservation of $E(u)$. In fact a direct computation, see the proof of Lemma \ref{lem cons weak}, 
allows one to further prove the conservation of $F(u)$.
\end{remark}

The next result holds for the initial datum $u_0$ in the natural energy space $H^1 \cap W^{1,4}$
without the sign condition on $m_0$.

\begin{theorem}[\cite{Chen18IUMJ}]\label{thm weak u}
Given $u_0 \in H^1 \cap W^{1,4}$. Then the Cauchy problem \eqref{Cauchy nov} admits
a unique global weak solution $u(t,\placeholder) \in H^1 \cap W^{1,4}$ for all $t \ge 0$.
Moreover, $E(u)$ is a conservation law.
\end{theorem}

\begin{remark}\label{rk choice}
For the instability analysis, we need to work with the initial datum $u_0$
in the restrictive function space $H^1 \cap W^{1,\infty}$
without the sign condition on $m_0$, for which neither Theorem \ref{thm weak m}
nor Theorem \ref{thm weak u} is applicable. One of the reasons is that weak solutions in $H^1 \cap W^{1,\infty}$
enjoy (spatial) Lipschitz regularity which suits well for the standard
theory for solvability of differential equations along the characteristics. The other reason is due to the fact that
while $E(u)$ conserves for the weak solutions in $H^1 \cap W^{1,4}$, $F(u)$ is only conserved for almost every $t > 0$ \cite{Chen18IUMJ}.
Although no previous local well-posedness theory has been developed for the Cauchy problem (\ref{Cauchy nov})
in $H^1 \cap W^{1,\infty}$, we will obtain the local well-posedness from the method of characteristics
under the assumption that our solution in $H^1 \cap W^{1,\infty}$ consists of a single peakon
perturbed by a single-peaked piecewise $C^1$ function, see Theorem \ref{theorem GWP}.
\end{remark}

Next we state the regularity of the nonlocal terms in \eqref{Cauchy nov}.
A similar argument as in \cite[Lemma 5]{DmitryPreprint} combined with the estimates in \cite[Section 2]{Chen18IUMJ} leads to

\begin{lemma}\label{lem cont Q}
If $u \in H^1 \cap W^{1,\infty}$, then $\Q[u] \in C(\R)$. If $u \in H^1 \cap C^1_0$, then $\Q[u] \in C^1_0$.
\end{lemma}

Following \cite{DmitryPreprint}, the function class we use here is $C^1_0$ which is suited for capturing the single peak in the
peaked solution $u$. Similarly to \cite[Lemma 6]{DmitryPreprint}, the location of the peak moves with its local characteristic speed.

\begin{lemma}\label{lem peak location}
Assume that there exists the unique weak solution $u \in C([0,T), H^1\cap W^{1,\infty})$
to the Cauchy problem \eqref{Cauchy nov} for some $T > 0$ with
a jump of $u_x$ across $x = a(t)$ such that $u(t, \cdot + a(t)) \in C^1_0$, $t \in [0,T)$. Then, we have $a \in C^1(0,T)$ and
$a'(t) = u^2(t,a(t))$, for $t \in [0,T)$.
\end{lemma}

Assuming local well-posedness of the Cauchy problem \eqref{Cauchy nov} for $u_0 \in H^1\cap W^{1,\infty}$, we
shall extend the result of Theorem A to prove the orbital stability of the single peakon $\varphi$ in $H^1$, 
see Theorem \ref{thm stab weak}.

\subsection{$H^1$-orbital stability of peakons for single-peaked perturbations}\label{subsec orbital}

Let us first recall the following characterization of $W^{1,p}$ functions in terms of the integrability of their spatial shifts.
\begin{theorem}[\cite{ziemer2012weakly} Theorem 2.1.6]\label{thm sobolev}
Let $1\le p < \infty$. Then $u\in W^{1,p}(\R^d)$ if and only if $u\in L^p(\R^d)$ and the quantity
\begin{equation*}
\int_{\R^d} \left|{u(x+h) - u(x) \over h} \right|^p \,dx 
\end{equation*}
remains bounded for all $h\in \R^d$.
\end{theorem}

We show now that the two functionals $E(u)$ and $F(u)$ are still conserved for the same weak solutions as
those assumed in Lemma \ref{lem peak location}.

\begin{lemma}\label{lem cons weak}
Assume that there exists the unique weak solution $u \in C([0,T), H^1\cap W^{1,\infty})$
to the Cauchy problem \eqref{Cauchy nov} for some $T > 0$.
Then, the values of $E(u)$ and $F(u)$ are conserved.
\end{lemma}

\begin{proof}
Rewrite the convolution form \eqref{novikov weak} of the Novikov equation as follows:
\begin{equation}\label{weak2}
u_t + u^2 u_x + \partial_x P_1(u, u_x) + P_2(u_x) = 0,
\end{equation}
where
\[
P_1(u, u_x) := {1\over2} \varphi \ast \left( {3\over2} uu_x^2 + u^3 \right), \quad P_2(u_x) := {1\over 4} \varphi \ast (u_x^3).
\]
Differentiating \eqref{weak2} in $x$ and using that $(1 - \partial_x^2) \varphi = 2\delta$ we obtain
\begin{equation}\label{weak2_x}
u_{xt} + (u^2u_x)_x - \left( {3\over2} uu_x^2 + u^3 \right) + P_1(u,u_x) + \partial_x P_2(u_x) = 0.
\end{equation}
For analysis of conservation laws, we will regularize the evolution equations \eqref{weak2} and \eqref{weak2_x}.
Let $\varepsilon>0$ and define
\[
\ub(x) := \eta_\varepsilon \ast u(x),
\]
where $\eta_\varepsilon(x) := {1\over \varepsilon} \eta\LC {x\over \varepsilon} \RC$ and $\eta \ge 0$ is a smooth even function compactly supported in a ball of radius 1, and with integral equal to 1.

Applying the mollifier $\eta_\varepsilon$ to \eqref{weak2} and using the cummutative and associative properties of the convolution, we obtain
\begin{equation}\label{reg weak}
\ub_t + \ub^2 \ux + \partial_x P_1(\ub, \ux) + P_2(\ux) = \Rr_1,
\end{equation}
where
\[
\Rr_1 := \ub^2 \ux - \overline{u^2u_x} + \partial_x P_1(\ub, \ux) - \partial_x \overline{P_1(u, u_x)} + P_2(\ux) - \overline{P_2(u_x)}.
\]
Similarly, from \eqref{weak2_x} we have
\begin{equation}\label{reg weak_x}
\ux_t + (\ub^2 \ux)_x - \LC {3\over2} \ub \ \ux^2 + \ub^3 \RC + P_1(\ub, \ux) + \partial_x P_2(\ux) = \Rr_2,
\end{equation}
where
\begin{align*}
\Rr_2 := \ & (\ub^2 \ux)_x - (\overline{u^2u_x})_x - \LC {3\over2} \ub\ \ux^2 + \ub^3 \RC + \overline{{3\over2} uu_x^2 + u^3} \\
& + P_1(\ub, \ux) - \overline{P_1(u, u_x)} + \partial_x P_2(\ux) - \partial_x \overline{P_2(u_x)}.
\end{align*}
We are now able to verify conservation of $E(u)$ and $F(u)$.

\medskip

 {\bf Conservation of $E(u)$:} Following \cite[Section 2]{Chen18IUMJ}, multiplying \eqref{reg weak} by $\ub$ and \eqref{reg weak_x} by $\ux$ we obtain a regularized local conservation law:
\begin{equation}\label{reg energy 1}
\partial_t \LC {\ub^2 + \ux^2 \over 2} \RC + \partial_x \LC{\ub^2 \ux^2 \over2} + \ub P_1(\ub,\ux) + \ub \partial_x P_2(\ux)  \RC = \ub \Rr_1 + \ux \Rr_2.
\end{equation}
Integration over $\R$ then gives
\begin{equation}\label{reg E}
\begin{split}
{1\over2}{d\over dt} E(\ub) & = \int_\R \LCB{1\over 3} \ub \LC \ub^3 - \overline{u^3} \RC_x + {1\over 3} \ux \LC \ub^3 - \overline{u^3} \RC_{xx} + {3\over 2} \ux \LC \ub\ \ux^2 - \overline{uu_x^2} \RC \right. \\
& \quad \quad \left. +\  \ux \LC \ub^3 - \overline{u^3} \RC + {1\over2} \ub \LC \ux^3 - \overline{u_x^3} \RC \RCB \,dx
\end{split}
\end{equation}

Note that $u \in H^1 \cap W^{1,\infty}$, and hence
$u, u_x \in L^p$ for any $2\le p \le \infty$. The properties of smooth approximation imply that
\begin{equation}\label{strong conv}
\|\ub - u\|_{L^p(\R)} \to 0, \quad \|\ux - u_x\|_{L^p(\R)} \to 0, \quad \text{as }\ \ \varepsilon \to 0.
\end{equation}
This way we know that the first, fourth, and fifth terms in the right-hand side of \eqref{reg E} all converge to zero as $\varepsilon \to 0$.

For the third term, note that we can write
\begin{align*}
\ub\ \ux^2 - \overline{uu_x^2} & = \ub \LC \ux^2 - \uxs \RC + \LC \ub\ \uxs - \overline{uu_x^2} \RC \\
& = \ub \LC \ux^2 - u_x^2 \RC + \ub \LC u_x^2 - \uxs \RC + \LC \ub\ \uxs - \overline{uu_x^2} \RC.
\end{align*}
The first two terms of the above can be treated using \eqref{strong conv}. For the last term, we can recall \cite[Lemma 3]{constantin2000global},
which states that if $f$ is uniformly continuous and bounded, and $\mu \in \mathcal{M}(\R)$, then
$\overline{f\mu} - \overline{f} \overline{\mu} \to 0$ in $L^1$.
Since $u$ is Lipschitz and $u_x^2 \in L^1$, we have that $\ub\ \uxs - \overline{uu_x^2} \to 0$ in $L^1$.
Therefore, the third term in the right-hand side of \eqref{reg E} converges to zero as $\varepsilon \to 0$.

Finally we look to show that
\[
\int_\R \ux \LC \ub^3 - \overline{u^3} \RC_{xx} \,dx = \int_\R \ub_{xxx} \LC \ub^3 - \overline{u^3} \RC \,dx \ \to \ 0 \quad \text{as }\ \varepsilon \to 0.
\]
Since the above obviously holds for smooth functions, one can use the Banach--Steinhaus theorem to observe that it is enough to show that $\LN \ub_{xxx} ( \ub^3 - \overline{u^3} ) \RN_{L^1}$ is uniformly bounded.

To this end, note that
\begin{equation}\label{u_xxx est}
\begin{split}
\LV \ub_{xxx}(x) \RV & = {1\over \varepsilon^4} \LV \int_{-\varepsilon}^{\varepsilon} u(x - y) \eta'''\LC {y \over \varepsilon} \RC \,dy \RV = {1\over \varepsilon^4} \LV \int_{-\varepsilon}^{\varepsilon} \LC u(x - y) - u(x) \RC \eta'''\LC {y \over \varepsilon} \RC \,dy \RV \\
& \le {1\over \varepsilon^4} \int_{-\varepsilon}^{\varepsilon} { |u(x - y) - u(x)| \over |y| } \LV \eta'''\LC {y \over \varepsilon} \RC \RV |y| \,dy \lesssim {1\over \varepsilon^2} \|u_x\|_{L^\infty}.
\end{split}
\end{equation}
It is also straightforward to check that
\begin{equation*}
\overline{u^3} - \ub^3 = \overline{r}_3(u) + (u - \ub)^3 + 3u \overline{r}_2(u) - 3u (u - \ub)^2,
\end{equation*}
where
\begin{equation*}
\overline{r}_n(u) := \int_{-\varepsilon}^{\varepsilon} \LC u(x - y) - u(x) \RC^n \eta_\varepsilon(y) \,dy
\end{equation*}
This way
\begin{equation*}
\begin{split}
\LV \overline{r}_n(u) (x) \RV & \le \int_{-\varepsilon}^{\varepsilon} \LC {|u(x - y) - u(x)| \over |y| } \RC^n \eta_\varepsilon(y) |y|^n \,dy \\
& \le \varepsilon^n \int_{-\varepsilon}^{\varepsilon} \LC {|u(x - y) - u(x)| \over |y| } \RC^n \eta_\varepsilon(y)  \,dy.
\end{split}
\end{equation*}
Moreover it follows from H\"older's inequality that
\begin{align*}
|\ub - u|^n & = \LV \int_{-\varepsilon}^{\varepsilon} \LC u(x - y) - u(x) \RC \eta_\varepsilon(y) \,dy \RV^n \\
& \le \LC \int_{-\varepsilon}^{\varepsilon} \eta_\varepsilon(y) \,dy \RC^{n-1} \LC \int_{-\varepsilon}^{\varepsilon} \LV u(x - y) - u(x) \RV^n \eta_\varepsilon(y) \,dy \RC \\
& = \LV \overline{r}_n(u) (x) \RV.
\end{align*}
An application of Fubini Theorem together with Theorem \ref{thm sobolev} implies that
\begin{equation}\label{r_n est}
\begin{split}
\LN \overline{r}_n(u) \RN_{L^1} & \le \varepsilon^n \int_{\R} \int_{-\varepsilon}^{\varepsilon} \LC {|u(x - y) - u(x)| \over |y| } \RC^n \eta_\varepsilon(y)  \,dy dx \\
& \le \varepsilon^n \int_{-\varepsilon}^{\varepsilon} \LB \int_{\R} \LC {|u(x - y) - u(x)| \over |y| } \RC^n \,dx \RB  \eta_\varepsilon(y)  \,dy \\
& \le \varepsilon^n \|u\|_{W^{1,n}}^n.
\end{split}
\end{equation}
From \eqref{u_xxx est} and \eqref{r_n est} it follows that
\begin{equation*}
\LN \ub_{xxx} ( \ub^3 - \overline{u^3} ) \RN_{L^1} \lesssim \|u_x\|_{L^\infty} \LC \varepsilon \|u\|_{W^{1,3}}^3 + \|u\|_{H^1}^2 \RC.
\end{equation*}
Putting together the above estimates we obtain that
\[
{d\over dt} E(\ub) \to 0, \quad \text{as} \quad \varepsilon \to 0,
\]
which proves the conservation of $E(u)$.

\medskip

{\bf Conservation of $F(u)$:} Similarly as before, to get the conservation law for $F(\ub)$ we multiply \eqref{weak2} by $4\ub^3 + 2\ub\ \ux^2$, multiply \eqref{weak2_x} by $-{4\over3} \ux^3 + 2\ub^2 \ux$ and integrate over $\R$ we have
\begin{align*}
{d\over dt} F(\ub) = \int_\R \LB (4\ub^3 + 2\ub\ \ux^2) \mathcal{R}_1 + \LC 2\ub^2 \ux -{4\over3} \ux^3 \RC \mathcal{R}_2 \RB \,dx.
\end{align*}
The rest of the proof follows in a similar way.
\end{proof}

Since the proof of Theorem A in \cite{Chen19} only makes use of the continuity of the solution
and the conservation of $E$ and $F$, we can recast the same idea in our current regularity setting
to obtain the following result.

\begin{theorem}\label{thm stab weak}
For every $0 < \varepsilon \ll 1$, let $u_0\in H^1 \cap C^1_0$ satisfy
\begin{equation*}
\|u_0 - \varphi\|_{H^1} < \varepsilon^4.
\end{equation*}
Assume existence of the unique weak solution $u \in C([0,T), H^1\cap W^{1,\infty})$
to the Cauchy problem \eqref{Cauchy nov} with the initial datum $u_0$ and the maximal existence time $T > 0$
such that $u(t, \cdot + a(t)) \in C^1_0$, $t \in [0,T)$ for some $a \in C^1([0,T)$ with $a(0) = 0$.
The corresponding solution $u$ satisfies
\begin{equation*}
\sup_{t\in [0, T)} \|u(t,\placeholder) - \varphi(\placeholder - a(t))\|_{H^1} < 2\LC 4 + \|u_{0x}\|_{L^\infty}^{1/2} \RC \varepsilon.
\end{equation*}
\end{theorem}

\begin{remark}
Because $u(t,\cdot + a(t)) \in C^1_0$ is $H^1$ close to $\varphi$ in Theorem \ref{thm stab weak},
it follows from continuous embedding of $H^1$ to $C^0$ and monotonicity of $\varphi$ with $\lim_{x \to 0^{\pm}} \varphi_x(x) = \mp 1$
that the location of the peak at $a(t)$ in Theorem \ref{thm stab weak} coincides with
the location of the maximum of $u$ at $\xi(t)$ in Theorem A.
\end{remark}

\subsection{Derivation of the evolution problem for perturbations to a single peakon}\label{subsec char nonlinear}

We shall construct a unique weak solution $u \in C([0,T), H^1\cap W^{1,\infty})$ to the Cauchy problem \eqref{Cauchy nov}
for some $T > 0$ with a single jump of $u_x$ across $x = a(t)$. We use the same decomposition \eqref{decomp}
and look for the modulation $a \in C^1(0,T)$ and the perturbation
$v \in C([0,T), H^1\cap W^{1,\infty})$ to the peakon $\varphi$.
If $v(t, \cdot) \in C^1_0$ for all $t\in [0,T)$, then the solution $u$ satisfies
$u(t,\cdot + a(t)) \in C^1_0$, $t \in [0,T)$ so that Lemma \ref{lem peak location} implies
that $a \in C^1(0,T)$ satisfies the following modulation equation:
\be\label{choice of s}
\dot a(t) = u^2 (t, a(t)) = (\varphi(0) + v(t,0))^2 = (1 + v(t,0))^2.
\ee
Note that the linear part of this modulation equation has already been used in the linearized
equation (\ref{dot-a-linear}). Thus, the problem of constructing the local
solution $u \in C([0,T), H^1\cap W^{1,\infty})$ is now replaced by the problem of constructing
the local solution $v \in C([0,T), H^1\cap W^{1,\infty})$ such that $v(t, \cdot) \in C^1_0$ for all $t\in [0,T)$.

Substituting (\ref{decomp}) and (\ref{choice of s}) into (\ref{Cauchy nov})
yields the following equation:
\be
v_t - (1 + v(t,0))^2 (\varphi_x + v_x) + \varphi^2 \varphi_x + (\varphi^2 v + \varphi v^2)_x + v^2 v_x + \mathcal{Q}[\varphi + v] = 0.
\ee
Canceling the stationary equation (\ref{stat-peakon}) for $\varphi$ and grouping the linear, quadratic, and cubic terms together,
we obtain the evolution equation for $v$ in the form:
\be
\label{v-t-N}
v_t + \mathcal{N}_1(v) + \mathcal{N}_2(v) + \mathcal{N}_3(v) = 0,
\ee
where
\begin{align*}
& \mathcal{N}_1(v) = \LB \varphi^2-1 \RB v_x + 2 \varphi_x \LB \varphi v - v(t,0) \RB +
{3\over2}\varphi_x \ast \LC  \varphi^2 v + {1\over2} \varphi_x^2 v + \varphi \varphi_x v_x \RC
+ {3\over4} \varphi \ast \LC \varphi_x^2 v_x \RC, \\
& \mathcal{N}_2(v) = 2\LB \varphi v - v(t, 0) \RB v_x + \varphi_x \LB v^2 - v^2(t,0) \RB + {1\over2} \varphi_x \ast \LC {3\over2}\varphi v_x^2 + 3\varphi_x vv_x + 3\varphi v^2 \RC + {3\over4} \varphi \ast \LC \varphi_x v_x^2 \RC \\
& \mathcal{N}_3(v) = \LB v^2 - v^2(t,0) \RB v_x + \mathcal{Q}[v].
\end{align*}
By Proposition \ref{prop sim v}, the linear part is reduced to the local form:
\begin{align}
\label{N-1}
\mathcal{N}_1(v) = \LB \varphi^2-1 \RB v_x - \varphi_x \LB \varphi v - v(t,0) \RB.
\end{align}
In order to simplify the quadratic part, we use the following proposition

\begin{proposition}\label{lem calc}
Let $f\in L^1(\R)$. Then
\be\label{sim calc}
\varphi_x \ast (\varphi f) + \varphi \ast (\varphi_x f) = -2\varphi \int^x_0 \varphi(y) f(y) \,dy.
\ee
\end{proposition}

\begin{proof}
Since $\varphi_x = -\text{sgn}(x) \varphi$, direct computation shows that
\begin{align*}
\varphi_x \ast (\varphi f) + \varphi \ast (\varphi_x f) & = e^{-x} \int_{-\infty}^x \LC \varphi_y - \varphi \RC f(y) \,dy + e^{x} \int^{\infty}_x \LC \varphi_y + \varphi \RC f(y) \,dy \\
& = e^{-x} \int_{0}^x \LC \varphi_y - \varphi \RC f(y) \,dy + e^{x} \int^{0}_x \LC \varphi_y + \varphi \RC f(y) \,dy \\
& = -2\varphi \int^x_0 \varphi(y) f(y) \,dy,
\end{align*}
which is \eqref{sim calc}.
\end{proof}

Using Proposition \ref{lem calc}, we prove the following proposition:

\begin{proposition}\label{prop sim v quadratic}
For $v\in H^1$ we have
\begin{align*}
& {1\over2} \varphi_x \ast \LC {3\over2}\varphi v_x^2 + 3\varphi_x vv_x + 3\varphi v^2 \RC + {3\over4} \varphi \ast \LC \varphi_x v_x^2 \RC \\
& = - \frac{3}{2} \varphi_x \LB v^2 - v^2(t,0) \RB - {3\over2} \varphi \int^x_0 \varphi (v^2 + v_y^2) \,dy.
\end{align*}
\end{proposition}

\begin{proof}
Integrating by parts and using (\ref{green}), we obtain
$$
\frac{3}{2} \varphi_x \ast \LC \varphi_x v v_x \RC = \frac{3}{4} \varphi \ast \LC \varphi_x v^2 \RC - \frac{3}{2} \varphi_x v^2
- \frac{3}{4} \varphi_x \ast \LC \varphi v^2 \RC + \frac{3}{2} \varphi_x v^2(t,0).
$$
Combining with other convolution terms, we obtain
$$
{3\over4} \varphi_x \ast \LC \varphi (v^2 + v_x^2) \RC + {3\over4} \varphi \ast \LC \varphi_x (v^2 + v_x^2) \RC =
- {3\over2} \varphi \int^x_0 \varphi (v^2 + v_y^2) \,dy,
$$
where the result of Proposition \ref{lem calc} has been used.
\end{proof}

By Proposition \ref{prop sim v quadratic}, the quadratic part is reduced to the simple form:
\begin{align}
\label{N-2}
\mathcal{N}_2(v) = 2\LB \varphi v - v(t, 0) \RB v_x - \frac{1}{2} \varphi_x \LB v^2 - v^2(t,0) \RB
- {3\over2} \varphi \int^x_0 \varphi (v^2 + v_y^2) \,dy.
\end{align}
Putting (\ref{N-1}) and (\ref{N-2}) into (\ref{v-t-N}), we obtain the Cauchy problem for
the perturbation $v$ to the peakon $\varphi$ in the following form:
\be\label{nonlinear v}
\left\{\begin{split}
& v_t + \LB (\varphi + v)^2 - (1 + v(t,0))^2 \RB v_x - \varphi_x \LC \varphi v - v(t,0) \RC - {1\over2} \varphi_x \LC v^2 - v^2(t,0) \RC \\
& \qquad  - {3\over2} \varphi \int^x_0 \varphi (v^2 + v_y^2) \,dy + \Q[v] = 0,\\
& v(0,x) = v_0(x).
\end{split}\right.
\ee
As discussed above, the small initial datum $v_0$ belongs to the space $H^1 \cap C^1_0$
and we are looking for the unique local weak solution $v \in C([0,T), H^1 \cap W^{1,\infty})$
to the evolution problem (\ref{nonlinear v}) for some $T > 0$ such that $v(t, \cdot) \in C^1_0$ for all $t\in [0,T)$.
The following result states local well-posedness of the Cauchy problem \eqref{nonlinear v}.

\begin{theorem}\label{theorem GWP}
For every initial datum $v_0 \in H^1 \cap C^1_0$, there exist the maximal existence time $T > 0$ and
the unique solution $v \in C([0,T),H^1 \cap C^1_0)$ to the Cauchy problem (\ref{nonlinear v})
that depends continuously on the initial datum $v_0 \in H^1 \cap C^1_0$.
\end{theorem}

Theorem \ref{theorem GWP} is proven next by using the method of characteristics.

\subsection{Solution to the evolution problem}
\label{subsec soln char}

The evolution problem \eqref{nonlinear v} suggests us to work with the characteristics $q(t,s)$
which satisfy the following evolution problem:
\be\label{char nonlinear}
\begin{cases}
\displaystyle {d q \over dt} = \LB \varphi(q) + v(t,q) \RB^2 - \LB 1 + v(t,0) \RB^2, \\
q(0, s) = s.
\end{cases}
\ee
Compared to the linearized evolution problem \eqref{char}, we cannot solve the nonlinear evolution
problem \eqref{char nonlinear} explicitly. However, we can analyze if the slope function
\be
\label{slope-function}
f(t,q) := \LB \varphi(q) + v(t,q) \RB^2 - \LB 1 + v(t,0) \RB^2
\ee
defines a well-posed initial-value problem in the correct solution space for $v$, as is done in the following lemma.

\begin{lemma}
\label{lem-nonlinear-1}
Assume that $v \in C([0,T),H^1 \cap C^1_0)$ with some maximal existence time $T > 0$.
There exists the unique solution $q\in C^1([0,T),C^1_0)$ to system \eqref{char nonlinear}
such that the mapping $\mathbb{R} \ni s \mapsto q(t,\cdot) \in C^1_0$ is invertible
for every $t \in [0,T)$ and satisfies $q(t,0) = 0$ and $\lim_{|s| \to \infty} q_s(t,s) = 1$.
\end{lemma}

\begin{proof}
Since $\varphi \in C^1_0 \subset W^{1,\infty}$ and $v(t,\cdot) \in C^1_0 \subset W^{1,\infty}$ for $t \in [0,T)$, then
$f$ is Lipschitz in $q$ and continuous in $t$ for every $t \in [0,T)$. By existence, uniqueness, and continuous dependence
theory for differential equations, the initial-value problem \eqref{char nonlinear} admits the unique solution $q(t,s)$
satisfying $q(\cdot, s) \in C^1(0,T)$ for any $s \in \R$ and
$q(t,\cdot) \in C^1_0$ for any $t \in [0,T)$. Moreover, $f(t,0) = 0$,
hence $q(t,0) = 0$ holds for all $t \in [0,T)$.

Differentiating the initial-value problem (\ref{char nonlinear}) with respect to $s$ piecewise for
$s > 0$ and $s < 0$ yields
\be\label{char nonlinear derivative}
\begin{cases}
\displaystyle {d q_s \over dt} = 2 \LB \varphi(q) + v(t,q) \RB \LB \varphi_x(q) + v_x(t, q) \RB q_s, \\
q_s(0, s) = 1,
\end{cases}
\quad s \in \R \backslash \{0\},
\ee
with the unique solution for every $s \in \R \backslash \{0\}$:
\be\label{q_s exp}
q_s(t,s) = \text{exp} \LC 2 \int^t_0 \LB \varphi(q) + v(\tau,q) \RB \LB \varphi_x(q) + v_x(\tau,q) \RB \,d\tau \RC > 0,
\ee
hence $q(t,\cdot)$ is invertible on $\R$ for $t \in [0,T)$. Moreover we have $\lim_{|s| \to \infty} q_s(t,s) = 1$ for $t \in [0,T)$ because $v_x(t,\cdot) \in L^{\infty}$ and
$v(t,q) \to 0$ as $|q| \to \infty$ for $t \in [0,T)$
thanks to the Sobolev embedding of $H^1(\R)$ to the space of continuous and decaying functions.
\end{proof}

Setting $V(t, s) := v(t, q(t,s))$ as in (\ref{V-on-characteristics}),
then it follows from \eqref{nonlinear v} that evolution of $V$ along the characteristics $q$
is given by
\be\label{nonlinear V}
\begin{cases}
\displaystyle{dV \over dt} = \varphi_x(q) \LB \varphi(q) V - V(t,0) \RB +  {1\over2} \varphi_x(q) \LB V^2 - V^2(t,0) \RB \\
\displaystyle\qquad \ \  + {3\over2} \varphi(q) \int^q_0 \varphi(v^2 + v_y^2) \,dy - \Q[v](q) \\
V(0,s) = v_0(s).
\end{cases}
\ee
Denote $V^0(t) := V(t,0) = v(t,0)$, where the last equality follows from $q(t,0) = 0$.
It follows from the initial-value problem \eqref{nonlinear V} as $s \to 0$ from either side
that $V^0$ satisfies the limiting initial-value problem
\be
\label{V-0-dynamics}
\begin{cases}
\displaystyle {dV^0 \over dt} = - \Q[v](0), \\
V^0(0) = v_0(0).
\end{cases}
\ee

In order to control solvability of the solution in (\ref{q_s exp}), we need to control $v_x$, and hence $V_s$
along the characteristics. Therefore we need to differentiate \eqref{nonlinear v} in order to
derive the evolution equation for $v_x$. The appearance of $\varphi'$ in \eqref{nonlinear v}
presents severe trouble when differentiating. The way to overcome that is to ``cut out" the origin
and consider solving the evolution equation for $w := v_x$ separately on $\R^+$ and $\R^-$.
This agrees with Lemma \ref{lem-nonlinear-1}, which suggests that for the solution $v \in C([0,T),H^1 \cap C^1_0)$ 
the spatial domain $\R$ can be partitioned into $\R^+$ and $\R^-$ on two sides from the peaked wave $\varphi$
invariantly in time $t$.

Computing derivative of \eqref{nonlinear v} separately on $\R^+$ and $\R^-$ and using the fact that $\varphi'' = \varphi$ on
$\R \backslash\{0\}$, we derive the evolution equation for $x \neq 0$:
\be\label{nonlinear w}
\left\{\begin{split}
& w_t + \LB (\varphi + v)^2 - (1 + v(t,0))^2 \RB w_x - \varphi \LC \varphi v - v(t,0) \RC - {1\over2} \varphi \LC v^2 - v^2(t,0) \RC \\
& \qquad  + \varphi \varphi_x w  - \varphi^2 v + \varphi_x v w + 2 \varphi w^2 + \frac{1}{2} v w^2 - v^3  \\
& \qquad -{3\over 2} \varphi^2 (v^2 + w^2) - {3\over2} \varphi_x \int^x_0 \varphi (v^2 + w^2) \,dy + \Pa[v] = 0,\\
& w(0,x) = v_{0x}(x),
\end{split}\right.
\ee
where
\be\label{def P}
\Pa[v](x) := {1\over2}\varphi \ast \LC {3\over2} v v_x^2 + v^3 \RC + {1\over4}\varphi_x \ast v_x^3.
\ee

Setting $W(t,s) := v_x(t, q(t,s))$ as in (\ref{W-on-characteristics}), then it follows that $W$ satisfies \eqref{relation W}.
If the mapping $\mathbb{R} \ni s \mapsto q \in C^1_0$ is invertible as in Lemma \ref{lem-nonlinear-1}, we have $\|V\|_{L^\infty} = \|v\|_{L^\infty}$ and
$\|W\|_{L^\infty} = \|v_x\|_{L^\infty}$. Writing the evolution problem (\ref{nonlinear w}) at the characteristics
yields for $s \neq 0$:
\be\label{nonlinear W}
\begin{cases}
\displaystyle{dW \over dt} = \varphi(q) \LB \varphi(q) V - V(t,0) \RB + {1\over2} \varphi_x(q) \LB V^2 - V^2(t,0) \RB \\
\displaystyle\qquad \ \ -\varphi(q) \varphi_x(q) W + \varphi(q)^2 V -\varphi_x(q) VW - 2\varphi(q) W^2 - {1\over2} VW^2 + V^3 \\
\displaystyle\qquad \ \  + {3\over2} \varphi(q)^2(V^2 + W^2) + {3\over2} \varphi_x(q) \int^q_0 \varphi(v^2 + w^2) \,dy  - \Pa[v](q), \\
W(0,s) = v_{0x}(s).
\end{cases}
\ee
Compared to the linearized evolution problem \eqref{char-V}, we cannot solve the nonlinear evolution
problems \eqref{nonlinear V} and \eqref{nonlinear W} explicitly. Nevertheless, we can analyze the vector
field for the evolution system
\begin{equation}
\label{nonlinear-F}
\frac{d}{dt} \left[ \begin{array}{l} q \\ V  \\ W \end{array} \right] = \left[ \begin{array}{l}
f^{(q)}(q,V) \\ f^{(V)}(q,V,W) \\ f^{(W)}(q,V,W) \end{array} \right] = : F(q,V,W),
\end{equation}
where components of $F(q,V,W)$ are given by
\begin{eqnarray*}
f^{(q)}(q,V,) & := & \LB \varphi(q) + V \RB^2 - \LB 1 + V^0 \RB^2,\\
f^{(V)}(q,V,W) & := & \varphi_x(q) \LB \varphi(q) V - V^0 \RB +  {1\over2} \varphi_x(q) \LB V^2 - (V^0)^2 \RB \\
& \phantom{t} &  + {3\over2} \varphi(q) \int^q_0 \varphi(v^2 + w^2) \,dy - \Q[v](q),\\
f^{(W)}(q,V,W) & := & \varphi(q) \LB \varphi(q) V - V^0 \RB + {1\over2} \varphi_x(q) \LB V^2 - (V^0)^2 \RB \\
& \phantom{t} & -\varphi(q) \varphi_x(q) W + \varphi(q)^2 V -\varphi_x(q) VW - 2\varphi(q) W^2 - {1\over2} VW^2 + V^3 \\
& \phantom{t} &  + {3\over2} \varphi(q)^2 (V^2 + W^2) + {3\over2} \varphi_x(q) \int^q_0 \varphi(v^2 + w^2) \,dy  - \Pa[v](q).
\end{eqnarray*}
The dynamical system (\ref{nonlinear-F}) is equipped with the initial datum:
\begin{equation}
\label{initial-datum-F}
\left[ \begin{array}{l} q \\ V  \\ W \end{array} \right] \biggr|_{t = 0} = \left[ \begin{array}{l}
s \\ v_0(s) \\ v_{0x}(s) \end{array} \right], \quad s \in \mathbb{R}.
\end{equation}
Because of the nonlocal terms in $f^{(V)}$ and $f^{(W)}$,
the vector field $F(q,V,W)$ computed for solutions to the dynamical system (\ref{nonlinear-F}) with
the initial datum (\ref{initial-datum-F}) with one value of $s \in \mathbb{R}$
requires global information about solutions $(q,V,W)$ computed for all other values of $s$ on $\mathbb{R}$.

The nonlocal terms are treated with the chain rule $v(q(s)) = V(s) $ and $v_x(q(s)) = W(s)$
provided that the mapping $\mathbb{R} \ni s \mapsto q \in C^1_0$ is invertible.
In addition, we use $V^0 = V(0)$. The following lemma show that the vector field $F(q,V,W)$
is locally Lipschitz with respect to $(q,V,W)$ and preserves properties of the mapping
$\mathbb{R} \ni s \mapsto q \in C^1_0$ and properties of the solution $(v,w)$.

\begin{lemma}
\label{lem-nonlinear-2}
For every $q \in C^1_0$ satisfying $q(0) = 0$, $\inf_{s \in \R} q_s(s) > 0$, and $\lim_{|s| \to \infty} q_s(s) = 1$
and every $v \in H^1 \cap C^1_0$,  the vector field $F(q,V,W)$ is locally Lipschitz in $(q,V,W)$ separately
for $q \in \R^+$ and $q \in \R^-$. Moreover, we have
\begin{itemize}
\item[(i)] $f^{(q)}(0,V^0) = 0$, $f^{(V)}(0,V^0,W) = -\Q[v](0)$,
\item[(ii)] $f^{(V)}(q(\cdot),V(\cdot),W(\cdot)) \in L^2$, $f^{(W)}(q(\cdot),V(\cdot),W(\cdot)) \in L^2$,
\item[(iii)] $\partial_s f^{(q)}(q(s),V(s)) = G(s) q_s(s)$ with $G \in L^{\infty}$ satisfying
$\lim_{|s| \to \infty} G(s) = 0$.
\end{itemize}
\end{lemma}

\begin{proof}
Thanks to the assumption $q_s(s) > 0$ for every $s \in \R$, the mapping $\R \ni s \mapsto q \in C^1_0$ is invertible, hence
$V(s) = v(q(s))$ belongs to $C^1_0$ and $W(s) = v_x(q(s))$ is bounded and continuous for $s \in \R^+$ and $s \in \R^-$.
Thanks to the assumption $\lim_{|s| \to \infty} q_s(s) = 1$ and the chain rule, it follows
from $v \in H^1$ that $V \in L^2$ and $W \in L^2$.
Thanks to the assumption $q(0) = 0$, the vector field
$F(q,V,W)$ in system (\ref{nonlinear-F}) can be considered separately for $q \in \R^+$ and $q \in \R^-$.

All local terms in $F(q,V,W)$ are locally Lipschitz in $(q,V,W)$ separately for $q \in \R^+$ and $q \in \R^-$.
The nonlocal terms in $f^{(V)}(q,V,W)$ are also locally Lipschitz in $(q,V,W)$ for every $q \in \R$, $V \in L^2$, and
$W \in L^2$, thanks to integrability of $v^2 + w^2$, invertibility of the mapping $\R \ni s \mapsto q \in C^1_0$, and the chain rule, e.g.
$$
\varphi(q) \int^q_0 \varphi(v^2 + w^2) \,dy = \varphi(q) \int_0^q \varphi(q(s')) (V^2 + W^2)(s') q_s(s') ds'
$$
and
\begin{align*}
\Q[v](q) & = {1\over 2}  \int_\R \varphi_x\big(q - q(s')\big)\LC {3\over2} VW^2 + V^3 \RC(s') q_s(s')\,ds' \\
& \quad + {1\over4} \int_\R \varphi\big(q - q(s')\big) W^3(s') q_s(s') \,ds',
\end{align*}
Similarly, it follows that the nonlocal terms in $f^{(W)}(q,V,W)$ are locally Lipschitz in $(q,V,W)$
for every $q \in \R$, $V \in L^2$, and $W \in L^2$.

It remains to verify items (i), (ii), and (iii).  It follows from the factorization formula:
$$
f^{(q)}(q,V) = ( \varphi(q) + 1 + V + V^0) ( \varphi(q) - 1 + V - V^0),
$$
that $f^{(q)}(q,V)$ is locally Lipschitz at $q = 0$ and $V = V^0$ with $f^{(q)}(0,V^0) = 0$.
Similarly, $f^{(V)}(q,V,W)$ is locally Lipschitz at $q = 0$, $V = V^0$, and every $W \in \R$
with $f^{(V)}(0,V^0,W) = -\Q[v](0)$. This verifies item (i).
Note that $f^{(W)}(q,V,W)$ is not locally Lipschitz at $q = 0$, $V = V^0$, and $W \neq 0$
because of the local terms $-\varphi(q) \varphi_x(q) W$ and $-\varphi_x(q) V W$ in $f^{(W)}(q,V,W)$.

For item (ii), all local terms in $f^{(V)}(q(\cdot),V(\cdot),W(\cdot))$ and $f^{(W)}(q(\cdot),V(\cdot),W(\cdot))$
are in $L^2$ because $\varphi, \varphi_x, V, W \in L^2 \cap L^{\infty}$. Similarly, nonlocal terms
are in $L^2$ because of invertibility of the mapping $\R \ni s \mapsto q \in C^1_0$ and the chain rule.
For instance, we have for $f^{(V)}$,
$$
\| \varphi(q(\cdot)) \int^{q(\cdot)}_0 \varphi(v^2 + w^2) \,dy \|_{L^2} \leq
\frac{1}{\left[\inf_{s \in \R} |q_s(s)|\right]^{1/2}} \| \varphi \|_{L^2} \| \varphi \|_{L^{\infty}} \| v \|_{H^1}^2
$$
and
$$
\| \Q[v](q(\cdot)) \|_{L^2} \leq
\frac{1}{\left[\inf_{s \in \R} |q_s(s)|\right]^{1/2}} \left( \frac{3}{4} \| \varphi_x \|_{L^2} \| v \|_{L^{\infty}} \| v \|_{H^1}^2
+ \frac{1}{4} \| \varphi \|_{L^2} \| w \|_{L^{\infty}} \| w \|_{L^2}^2 \right),
$$
and similar estimates for $f^{(W)}$.

Finally, for item (iii), we have explicitly
$$
\partial_s f^{(q)}(q(s),V(s)) = 2 (\varphi(q(s)) + V(s)) (\varphi_x(q(s)) + W(s)) q_s(s) =: G(s) q_s(s),
$$
so that $G \in L^{\infty}$ and $\lim_{|s| \to \infty} G(s) = 0$.
\end{proof}

Theorem \ref{theorem GWP} is proven by using Lemma \ref{lem-nonlinear-2}.

\begin{proof1}{\em of Theorem \ref{theorem GWP}.}
We consider the initial datum $v_0 \in H^1 \cap C^1_0$ for which $v_{0x} \in L^2$ is continuous
separately for $x \in \R^+$ and $x \in \R^-$. The dynamical system (\ref{nonlinear-F}) is considered
with the initial datum (\ref{initial-datum-F}) which satisfies the assumptions of Lemma \ref{lem-nonlinear-2}.

By Lemma \ref{lem-nonlinear-2}, the vector field preserves the assumptions in the sense that if
we define
\begin{align*}
\left\{ \begin{array}{l}
\hat{q}(t,s) = s + \int_0^t f^{(q)}(q(t',s),V(t',s)) dt', \\
\hat{V}(t,s) = v_0(s) + \int_0^t f^{(V)}(q(t',s),V(t',s),W(t',s)) dt', \\
\hat{W}(t,s) = v_{0x}(s) + \int_0^t f^{(W)}(q(t',s),V(t',s),W(t',s)) dt',
\end{array} \right.
\end{align*}
and
$$
\hat{q}_s(t,s) = 1 + \int_0^t G(s) q_s(t',s) dt',
$$
then for every $t$ on a compact interval $[-\tau,\tau]$ with small $\tau > 0$,
we have $\hat{q} \in C^1_0$ satisfying $\hat{q}(0) = 0$, $\inf_{s \in \R} \hat{q}_s(s) > 0$, and $\lim_{|s| \to \infty} \hat{q}_s(s) = 1$
and $\hat{v} \in H^1 \cap C^1_0$. By the existence and uniqueness theory for differential equations,
there exists the unique solution $q \in C^1([0,T),C^1_0)$,
$V \in C^1([0,T), H^1 \cap C^1_0)$, and $W \in C^1([0,T),C^0(\R^+) \cap C^0(\R^-))$
to system (\ref{nonlinear-F}) for some maximal existence time $T > 0$.
The solution depends continuously on the initial data and preserves
invertibility of the mapping $\R \ni s \mapsto q \in C^1_0$ with $q(t,0) = 0$, $\inf_{s \in \R} q_s(t,s) > 0$,
and $\lim_{|s| \to \infty} q_s(t,s) = 1$. Therefore, the transformation
formulas $V(t,s) = v(t,q(t,s))$ and $W(t,s) = w(t,q(t,s))$ are invertible
and the solutions $(q,V,W)$ yields the unique solution $v \in C^1([0,T),H^1 \cap C^1_0)$
to the evolution problem (\ref{nonlinear v}).

Continuous dependence of the solution $v \in C^1([0,T),H^1 \cap C^1_0)$ on the initial datum
$v_0 \in H^1 \cap C^1_0$ is obtained from the continuous dependence theory for differential equations thanks to
the Lipschitz continuity of the vector field $F(q,V,W)$ in Lemma \ref{lem-nonlinear-2}.
\end{proof1}

\subsection{Proof of instability}\label{subsec instab}

The characteristics $q = 0$ at $s = 0$ is the breaking point for the initial-value problem \eqref{nonlinear W}
since $W$ may have a jump discontinuity across $s = 0$. This point corresponds to the peak's location
for a perturbed single peakon, according to the decomposition (\ref{decomp}).
As follows from the proof of Theorem \ref{theorem GWP}, 
the dynamical system (\ref{nonlinear-F}) admits the unique solution in the form
$W \in C^1([0,T],C^0(\R^+) \cap C^0(\R^-))$. Therefore, we can define the one-sided limits $W^0_\pm \in C^1(0,T)$ by
\begin{equation*}
W^0_{\pm}(t): = \lim_{s \to 0^{\pm}} W(t, s) = \lim_{s\to 0^\pm} v_x(t, q(t,s)),
\end{equation*}
which satisfies the initial value problems
\be
\label{W-0-dynamics}
\left\{ \begin{array}{l} \displaystyle {dW^0_\pm \over dt} = \pm \LC 1+ V^0 \RC W^0_\pm + V^0 - 
{1\over2} \LC 1 + V^0 \RC \LC W^0_\pm \RC^2 + {3\over2} \LC V^0 \RC^2 + \LC V^0 \RC^3 - \Pa[v](0), \\\\
W_{\pm}^0(0) = v_{0x}(0^{\pm}).\end{array} \right.
\ee
This initial-value problem is combined with (\ref{V-0-dynamics}) which determines evolution of $V^0$.
The following lemma gives estimates for the two nonlocal terms in (\ref{V-0-dynamics}) and (\ref{W-0-dynamics}).

\begin{lemma}\label{lem est PQ}
Let the assumptions of Theorem \ref{thm stab weak} hold and define 
$v(t,\cdot) := u(t,\cdot + a(t)) - \varphi$ with $v \in C([0,T),H^1 \cap C^1_0)$. 
There exists $\varepsilon_0 > 0$ and $C_0 > 0$ such that for every $\varepsilon \in (0,\varepsilon_0)$ 
we have for every $x \in \R$ and every $t \in [0, T)$,
\begin{equation}
\label{control-P-Q}
\big| \Pa[v](t,x) + \Q[v](t,x) \big| < C_0 \varepsilon^2 (1 + \|u_{0x}\|_{L^\infty}^{3/2} + \varepsilon \|u_{0x}\|_{L^\infty}^2).
\end{equation}
\end{lemma}

\begin{proof}
By Theorem \ref{thm stab weak}, it follows for $\varepsilon$ small enough that 
\begin{equation}\label{size v}
\| v \|_{H^1} < 2 \LC 4 + \|u_{0x}\|_{L^\infty}^{1/2} \RC \varepsilon.
\end{equation}
Since $\|\varphi\|_{L^\infty} = \|\varphi_x\|_{L^\infty} = 1$, similar to \cite[(2.7)--(2.8)]{Chen18IUMJ} we obtain
\begin{equation*}
\begin{split}
\|v_x\|_{L^4}^4 & = 3 \int_\R (v^4 + 2v^2 v_x^2)\,dx - 3 F(v) \\
& \le 3 \|v\|_{L^\infty} \int_\R (v^2 + 2v_x^2)\,dx - 3 F(v) \\
& \le 3 \Big( 2\| v \|_{H^1}^4 - F(v) \Big),
\end{split}
\end{equation*}
indicating that $F(v) \le 2\| v \|_{H^1}^4$, where $F(v)$ is defined in \eqref{cons law 2}.  Interpolation implies that
\[
\|v_x\|_{L^3}^3 \le \sqrt{3} \|v\|_{H^1}\sqrt{2\|v\|_{H^1}^4 - F(v)},
\]
and hence
\begin{eqnarray}
\nonumber
\LV \Pa[v] + \Q[v] \RV & = & 
\LV {1\over2}\LC \varphi + \varphi_x \RC \ast \LC {3\over2} v v_x^2 + v^3 \RC + {1\over4}\LC \varphi + \varphi_x \RC \ast v_x^3 \RV\\
\nonumber
& \le & \LN {3\over2} v v_x^2 + v^3 \RN_{L^1} + {1\over2} \|v_x^3\|_{L^1}\\
\nonumber
& \le & {3\over2}\|v\|_{L^\infty} \|v\|^2_{H^1} + {1\over2} \|v_x\|^3_{L^3} \\
& \le & {3\over2}\|v\|_{H^1}^3 + {\sqrt{3}\over2} \|v\|_{H^1}\sqrt{2\|v\|_{H^1}^4 - F(v)}.
\label{est PQ}
\end{eqnarray}
Plugging $u = \varphi + v$ into $F(u)$ and
using $\|\varphi\|_{L^2} = \|\varphi_x\|_{L^2} = 1$, we obtain
\begin{align*}
|F(v)| & \le \LV F(u) - F(\varphi) \RV \\
& \quad + 2\LV \int_\R \LC 2v^2v_x\varphi_x + v^2\varphi_x^2 + 2vv_x^2\varphi + 4vv_x\varphi\varphi_x + 2v\varphi\varphi_x^2 + v_x^2\varphi^2 + 2v_x\varphi^2\varphi_x \RC\,dx \RV \\
& \quad + \LV \int_\R \LC 4v^3\varphi + 6v^2\varphi^2 + 4 v\varphi^3 \RC\,dx \RV + {1\over3} \LV \int_\R \LC 4v_x^3\varphi_x +6v_x^2\varphi_x^2 + 4v_x\varphi_x^3 \RC\,dx \RV\\
& \le \LV F(u) - F(\varphi) \RV + {4\over3} \LV \int_\R v_x^3\varphi_x \,dx \RV + \LC12+{4\over3} \RC \|v\|_{H^1} + 20 \|v\|_{H^1}^2 + 10\|v\|_{H^1}^3.
\end{align*}
Note that we have
\[
\LV {4\over3} \int_\R v_x^3\varphi_x \,dy \RV \le {4\over3} \|v_x\|_{L^3}^3 \le {4\over\sqrt{3}}\|v\|_{H^1}\sqrt{2\|v\|_{H^1}^4 - F(v)}.
\]
Thus, for $ \|v\|_{H^1} \ll 1$ sufficiently small it follows that
\be\label{est F}
|F(v)| \le \LV F(u) - F(\varphi) \RV + {4\over\sqrt{3}}\|v\|_{H^1}\sqrt{2\|v\|_{H^1}^4 - F(v)} + 15 \|v\|_{H^1}.
\ee
Thanks to the conservation $F(u) = F(u_0)$, a direct calculation yields that 
\begin{align*}
 \LV F(u) - F(\varphi) \RV & = \LV F(u_0) - F(\varphi) \RV \\
 & \le \LV \int_\R \LC u_0^4 - \varphi^4 \RC\,dx \RV + 2\LV \int_\R \LC u_0^2 u_{0x}^2 - \varphi^2 \varphi_x^2 \RC\,dx \RV + {1\over3}\LV \int_\R \LC u_{0x}^4 - \varphi_x^4 \RC\,dx \RV.
\end{align*}
Following \cite[Lemma 2.4]{Chen19}, we estimate the above as follows:
\begin{eqnarray*}
\LV \int_\R \LC u_0^4 - \varphi^4 \RC\,dx \RV & \le & \|v_0\|_{L^\infty} \|u_0 + \varphi\|_{L^\infty} (\|u_0\|_{L^2}^2 + \|\varphi\|_{L^2}^2) \\
& \le & \|v_0\|_{H^1} \big(\|v_0\|_{H^1} + 2 \big) \big(\|v_0\|_{H^1}^2 + 2\|v_0\|_{H^1} + 2 \big),\\
\LV \int_\R \LC u_0^2 u_{0x}^2 - \varphi^2 \varphi_x^2 \RC\,dx \RV & \le & \|v_0\|_{L^\infty} \|u_0 + \varphi\|_{L^\infty} \|u_{0x}\|_{L^2}^2 + \|\varphi\|_{L^\infty}^2 \|v_{0x}\|_{L^2} \|u_{0x} + \varphi_x\|_{L^2}\\
& \le & \|v_0\|_{H^1} \big( \|v_0\|_{H^1} + 2 \big) \big(\|v_0\|_{H^1} + 1\big)^2 + \|v_0\|_{H^1} \big( \|v_0\|_{H^1} + 2 \big),
\end{eqnarray*}
and
\begin{eqnarray*}
\LV \int_\R \LC u_{0x}^4 - \varphi_x^4 \RC\,dx \RV & \le & \LC \int_\R (u_{0x}^2 + \varphi_x^2)^2(u_{0x} + \varphi_x)^2 \,dx \RC^{1/2} \|v_{0x}\|_{L^2} \\
& \le & 3 \LC \int_\R (u_{0x}^6 + \varphi_x^6) \,dx \RC^{1/2} \|v_0\|_{H^1}  \le 3 \|v_0\|_{H^1} \sqrt{\|u_{0x}\|_{L^\infty}^4 \|u_{0x}\|_{L^2}^2 + {1\over3}} \\
& \le & \LC 3\|u_{0x}\|_{L^\infty}^2(\|v_0\|_{H^1} + 1) + \sqrt{3} \RC \|v_0\|_{H^1}.
\end{eqnarray*}
where we have used that $\|\varphi_x\|_{L^6}^6 = {1\over 3}$. Putting the above together yields
\begin{equation*}
\LV F(u) - F(\varphi) \RV \le \LC 2\|u_{0x}\|_{L^\infty}^2 + 15 \RC \|v_0\|_{H^1}.
\end{equation*}

Plugging this and \eqref{est v} into \eqref{est F} we have
\[
|F(v)| \le {4\over\sqrt{3}}\|v\|_{H^1}\sqrt{2\|v\|_{H^1}^4 - F(v)} + K,
\]
where $K := \LC 30+2\|u_{0x}\|_{L^\infty}^2 \RC \|v_0\|_{H^1}$. Solving the above we get
\begin{align*}
|F(v)| \le 6\|v\|_{H^1}^2 + 4 \|v\|_{H^1}^3 + \sqrt{6K} \|v\|_{H^1} + K.
\end{align*}
For $\varepsilon$ sufficiently small, we can find some large $C>0$ such that 
\[
|F(v)| \le  C \varepsilon^2 (1 + \|u_{0x}\|_{L^\infty} + \varepsilon^2 \|u_{0x}\|_{L^\infty}^2).
\]
Plugging this into \eqref{est PQ} and by further shrinking $\varepsilon$ if needed, we obtain \eqref{control-P-Q}.
\end{proof}

Theorem \ref{thm nonlin stab} is proven by using Theorem \ref{thm stab weak},
Theorem \ref{theorem GWP}, and Lemma \ref{lem-nonlinear-2}.

\begin{proof}{\em of Theorem \ref{thm nonlin stab}.}
By Theorem \ref{theorem GWP}, we consider the unique solution $v \in C([0,T),H^1 \cap C^1_0)$ to the
Cauchy problem \eqref{nonlinear v}. It follows from the bound (\ref{initial-bound-theorem})
and the decomposition \eqref{decomp} with $a(0) = 0$ that the initial datum $v_0 \in H^1 \cap C^1_0$
satisfies the bound
\begin{equation}
\label{initial-bound-proof}
\| v_0 \|_{H^1} + \| v_{0x} \|_{L^{\infty}} < \delta.
\end{equation}
Let $\varepsilon > 0$ be a small parameter to be determined below. By Theorem \ref{thm stab weak}, we have
\be\label{est v}
\text{if}\ \ \|v_0\|_{H^1} < \varepsilon^4, \ \text{ then }\ \ \|v(t,\placeholder)\|_{H^1} < 2 \LC 4 + \|u_{0x}\|_{L^\infty}^{1/2} \RC \varepsilon,
\ee
From \eqref{initial-bound-proof} we know that for $\delta$ sufficiently small,
\begin{equation}\label{newA}
\|u_{0x}\|_{L^\infty}^{1/2} < (1+\delta)^{1/2} < \sqrt{2}.
\end{equation}
Therefore Sobolev embedding implies that
\be\label{est V}
|V^0(t)| \le \|v(t,\placeholder)\|_{L^\infty} \le \|v(t, \placeholder)\|_{H^1} < ({8 + 2\sqrt{2}}) \varepsilon < 12\varepsilon.
\ee

{\bf Instability.}
The instability argument relies on the behavior of $v_{x}(t,x)$ near the peak at $x = 0$ from the right side,
where the linear instability result of Theorem \ref{thm lin stab} suggests at least exponential growth.
Therefore, picking $W^0_+$ in (\ref{W-0-dynamics}),
and using an integrating factor we obtain
\begin{align*}
\begin{split}
{d \over dt} \LB e^{-t}(V^0 + W^0_+) \RB & =  e^{-t}\LB {3\over2} \LC V^0 \RC^2 + V^0 W^0_+ - {1\over2} \LC 1 + V^0 \RC \LC W^0_+ \RC^2  + \LC V^0 \RC^3 - \Pa[v](0) - \Q[v](0) \RB\\
& \le e^{-t} \LB {5\over2} \LC V^0 \RC^2 - {1\over4} \LC 1 + 2V^0 \RC \LC W^0_+ \RC^2 + \LC V^0 \RC^3 - \Pa[v](0) - \Q[v](0) \RB.
\end{split}
\end{align*}
Therefore for $\varepsilon$ sufficiently small, it follows from \eqref{est V} that
\be\label{eq v+w new}
{d \over dt} \LB e^{-t}(V^0 + W^0_+) \RB  \le e^{-t} \LB 3 \LC V^0 \RC^2 - \Pa[v](0) - \Q[v](0) \RB.
\ee

Lemma \ref{lem est PQ} yields the control of $\Pa[v](0)$ and $\Q[v](0)$ in (\ref{control-P-Q}). By integrating \eqref{eq v+w new} and using
(\ref{control-P-Q}), \eqref{newA} and \eqref{est V}, we obtain
\be\label{eqn V}
V^0(t) + W^0_+(t) \le e^t \LB V^0(0) + W^0_+(0) + C \varepsilon^{2} \RB,
\ee
for some $C>0$. Let us pick
the initial datum $v_0 \in H^1\cap C^1_0$ satisfying $v_0(0) = 0$ and
\be
\label{derivative-large}
\lim_{x\to 0^+} v_{0x}(x) = -\|v_{0x}\|_{L^\infty} = -2C \varepsilon^{2}.
\ee
This is possible provided that for any given $\delta > 0$ in the initial bound (\ref{initial-bound-theorem}) (and hence \eqref{initial-bound-proof}),
the small parameter $\varepsilon > 0$ is chosen to satisfy the bound:
\[
\varepsilon^4 + 2C \varepsilon^{2} < \delta.
\]
Since $V^0(0) = 0$ and $W^0_+(0) = -2C \varepsilon^{2}$, we obtain from
\eqref{eqn V} that
\[
V^0(t) + W^0_+(t) \le - C \varepsilon^{2} e^t,
\]
which implies that
\[
|V^0(t) + W^+_0(t)| > 2 \quad \text{for} \quad t > t_0 := \log\LC 2 \over C \varepsilon^2 \RC > 0.
\]
Thanks to the bound (\ref{est V}) on $V^0(t)$, this implies that $|W^+_0(t)| > 1$ for $t > t_0$.

If $t_0 < T$, then we have the instability (\ref{final-bound-theorem}).
If $t_0 > T$, then $T$ is finite and we have $\| v_x(t,\cdot) \|_{L^{\infty}} \to \infty$ as $t \to T$
due to the fact that $\| v(t,\cdot) \|_{H^1}$ is bounded from the $H^1$ conservation of solutions.
In this case, the existence of another $t_0' \in (0,T)$ such that $\| v_x(t_0',\cdot) \|_{L^{\infty}} > 1$
follows from the continuity arguments.

\medskip

{\bf Blow-up.} Now we want to show that by choosing suitable initial datum satisfying (\ref{initial-bound-theorem}), the corresponding solution can indeed blow up in finite time.

Recall from (\ref{W-0-dynamics}) that we have
\begin{equation*}
\begin{split}
{dW^0_+\over dt} = -{1\over2}(1+V^0) (W^0_+ - 1)^2 + {1\over2} + {3\over2} V^0 + {3\over2} (V^0)^2 + (V^0)^3 - \Pa[v](0).
\end{split}
\end{equation*}
Note from \eqref{control-P-Q} and \eqref{est v}--\eqref{est V} that for $\varepsilon$ sufficiently small, $W^0_+$ satisfies the following Ricatti inequality
\begin{equation*}
{dW^0_+\over dt} \le -{1\over2}(1 - 12\varepsilon) (W^0_+ - 1)^2 + {1\over 2} + 20 \varepsilon.
\end{equation*}
Therefore it follows from the routine analysis of the differential inequality (see, for example \cite[Lemma 3.3]{Chen16JFA}) that if we choose initial datum satisfying
\begin{equation}\label{initial W}
W^0_+(0) < 1 - \sqrt{{1+40\varepsilon \over 1- 12\varepsilon}},
\end{equation}
then $W^0_+(t)$ tends to $-\infty$ in finite time.
To be more precise, let us pick the initial datum $v_0 \in H^1\cap C^1_0$ satisfying
\begin{equation*}
\|v_0\|_{H^1} < \varepsilon^4, \qquad \lim_{x \to 0^+} v_{0x}(x) = - 30 {\varepsilon},
\end{equation*}
with
\[
\varepsilon^4 + 30 {\varepsilon} < \delta.
\]
Then \eqref{initial W} is satisfied, and hence $v_x(t, 0) \to -\infty$ as $t \to T^*$ for some $T^* < \infty$. 
Hence the maximal existence time $T$ satisfies $T \leq T^* < \infty$.
\end{proof}

\vspace{0.25cm}

{\bf Acknowledgments:} The work of RMC is partially supported by National Science Foundation under grant DMS-1613375 and DMS-1907584.
The work of DEP is partially supported by the NSERC Discovery grant.

\bibliographystyle{siam}
\bibliography{Reference}

\end{document}